\documentclass{amsart}
\usepackage{amsmath}
\usepackage{amsfonts}
\usepackage{euscript}
\usepackage[toc,page]{appendix}
\usepackage{amssymb}
\usepackage{graphicx}
\usepackage{subcaption}
\usepackage{enumitem}

\usepackage{exscale}
\usepackage{cite}
\usepackage{epsfig}
\usepackage{amscd}
\usepackage{graphics}
\usepackage{color}

\renewcommand{\L}[1]{\EuScript{L}\{#1\}}

\usepackage{amsthm}
\usepackage[utf8]{inputenc}
\usepackage[english]{babel}

\newtheorem{theorem}{Theorem}
\newtheorem{corollary}[theorem]{Corollary}

\newtheorem{question}{Question}
\newtheorem{remark}[theorem]{Remark}

\def\reals{{\mathbb R}}

\def\supp{\mathrm{supp}\,}

\def\O{{\mathcal O}}

\def\phi{\varphi}
\def\half{{\frac{1}{2}}}

\def\be{\begin{eqnarray*}}
\def\ee{\end{eqnarray*}}
\def\ben{\begin{eqnarray}}
\def\een{\end{eqnarray}}

\def\L2R{L_{\text{Rest}}^2}

\def\11{\mathds{1}}

\def\L2c{L^2_{\text{comp}}}

\def\p{\partial}

\def\bu{\bar{u}}

\def\Area{\text{Area}}
\def\tOmega{\tilde{\Omega}}

\begin{document}

 \title[Energy Distribution]{Energy Distribution for Dirichlet
   Eigenfunctions on Right Triangles}

   \author[H. Christianson]{Hans Christianson}
\address[H. Christianson]{ Department of Mathematics, University of North Carolina.\medskip}
 \email{hans@math.unc.edu}
 
 \author[D. Pezzi]{Daniel Pezzi}
\address[D. Pezzi]{ Department of Mathematics, Johns Hopkins.\medskip}
 \email{dpezzi1@jhu.edu}
 
 \begin{abstract}

  In this paper, we continue the study of eigenfunctions on triangles initiated by the first author in \cite{Chr-tri} and \cite{Chr-simp}.  The Neumann data of Dirichlet eigenfunctions on triangles enjoys an equidistribution law, being equidistributed on each side.  The proof of this result is remarkably simple, using only the radial vector field and a Rellich type integrations by parts.  The equidistribution law, including on higher dimensional simplices, agrees with what Quantum Ergodic Restriction would predict. However, distribution of the Neumann data on subsets of a side is not well understood, and elementary methods do not appear to give enough information to draw conclusions.  
  
  In the present note, we first show that an ``obvious'' conjecture fails even for the simplest right isosceles triangle using only Fourier series.  We then use a result of Marklof-Rudnick \cite{Marklof-Rudnick} in which the authors show an interior {\it spatial} equidistribution law for a density one subsequence of eigenfunctions to give an estimate on energy distribution of eigenfunctions on the interior. Finally we present some numerical computations suggesting the behaviour of eigenfunctions on almost isosceles triangles is quite complicated.

 \end{abstract}

\maketitle

\section{Introduction}

 Eigenfunctions on bounded Euclidean domains are used to model many physical and mechanical phenomena, 
  and can be used to construct solutions to separable partial differential equations such as wave and Schr\"odinger type equations.  
  The study of eigenfunctions and solutions to wave type equations are so closely related that concepts like propagation and flow are often used to understand eigenfunctions.
  Since waves propagate along straight lines, it is reasonable to expect eigenfunctions to live along straight lines.   Since 
  waves meeting a smooth boundary  transversely  reflect according to Snell's law,  eigenfunctions have incoming and outgoing components as well. And since waves are very complicated near corners and other discontinuities, so are eigenfunctions.

In this note, we study the distribution of interior energy of eigenfunctions on right triangles, which is a measure of phase space concentration.  If the billiard flow on a reasonable domain is ergodic, then quantum ergodicity \cite{Shn,Zel1,CdV,ZZ} implies that a density one subsequence of the eigenfunctions equidistributes in both space and phase space.  This paper investigates similar properties on triangles, however no assumption about classical ergodicity is made, and instead the theoretical components of this paper rely on several previous results concerning the distribution of Neumann data mass on sides proved by the first author in \cite{Chr-tri} as well as the  spatial equidistribution result of Marklof-Rudnick \cite{Marklof-Rudnick} on rational polygons.   

There are several results in this paper.  First, we use integration by parts to connect interior energy to certain weighted boundary integrals.  A comparison with the results in \cite{Chr-tri}  suggests the eigenfunctions have nice phase space distribution properties, however this appears to be false.  In fact, even on the right isosceles triangle, the weighted boundary integrals are subtle, even though the phase space distribution follows from symmetry.  Second, using the spatial distribution for rational polygons in \cite{Marklof-Rudnick}, we prove that on rational right triangles there is a density one subsequence that has frequency localization estimates, but an asymptotic is unclear. For a sequence $a_j$, a density one subsequence is a subsequence $a_{j_k}$ such that 

\begin{equation}
\label{E:density-1}
    \lim_{N\rightarrow \infty}\frac{\#\{k: j_k\leq N\}}{N} = 1.
\end{equation}
Finally, we provide some numerical data which suggests the weighted boundary integrals do not have an asymptotic, or at least not for the whole sequence of eigenfunctions.

\subsection*{Acknowledgments}
The authors would like to thank Jeremy Marzuola and Idris Assani for reading early drafts of this work.  They would also like to thank the anonymous referee, whose comments and suggestions helped clarify several important parts of the proofs in this paper.  D.P. was supported in part by a UNC-Chapel Hill Summer Undergraduate Research Fellowship.

\section{Theoretical results on right triangles}

Our first set of results is on the right isosceles triangle.  

\begin{theorem}
\label{T:1}
Let $T$ be the right isosceles triangle in the $xy$-plane, oriented as 
\[
T = \{ 0 \leq x \leq 1, \,\, 0 \leq y \leq 1-x \}.
\]
Consider the Dirichlet eigenfunction problem on $T$:
\begin{equation}
\label{E:efcn1}
\begin{cases}
-h^2 \Delta u = u, \text{ on } T, \\
u|_{\p T} = 0, \\
\| u \|_{L^2(T)} = 1,
\end{cases}
\end{equation}
where $\Delta = \p_x^2 + \p_y^2$ and $h^{-2}$ denotes the eigenvalues of $-\Delta$, taking discrete values.  Then there is an orthonormal basis for $L^2$ of eigenfunctions satisfying
\[
\| h \p_x u \|^2_{L^2(T)} = \|h \p_y u \|^2_{L^2(T)} = 1/2.
\]
On the other hand, there exists a subsequence of these eigenfunctions whose Neumann data satisfies
\[
\liminf_{h \to 0} \int_{0 \leq x \leq 1/2 } | h \p_\nu u(x,0) |^2 dx > \limsup_{h \to 0} \int_{1/2 \leq x \leq 1 } | h \p_\nu u(x,0) |^2 dx.
\]

\end{theorem}

\begin{remark}
This theorem shows that this sequence of eigenfunctions has a subsequence which is not quantum ergodic on the boundary, even though the eigenfunctions are weakly equidistributed in phase space.

\end{remark}

Our second main result is about phase space distribution on rational right triangles.
  The result is heavily dependent on the choice of orientation for the triangle, and rotating the triangle changes the result.  This result is then meant merely as an example of what one can prove with elementary techniques.

Let $\Omega \subset \reals^2$ be a right triangle, and 
consider the semiclassical Laplace eigenfunction problem \eqref{E:efcn1}.

After rescaling and rotating, 
assume $\Omega$ is oriented so that it can be written as $ \Omega = \{
(x,y): 0 \leq x \leq a, \,\, 0 \leq y \leq 1-\frac{x}{a} \}$.  Let $F_1,F_2,F_3$ be
the sides as in Figure \ref{F:acute}.

 \begin{figure}
\hfill
\centerline{
\begingroup%
  \makeatletter%
  \providecommand\color[2][]{%
    \errmessage{(Inkscape) Color is used for the text in Inkscape, but the package 'color.sty' is not loaded}%
    \renewcommand\color[2][]{}%
  }%
  \providecommand\transparent[1]{%
    \errmessage{(Inkscape) Transparency is used (non-zero) for the text in Inkscape, but the package 'transparent.sty' is not loaded}%
    \renewcommand\transparent[1]{}%
  }%
  \providecommand\rotatebox[2]{#2}%
  \newcommand*\fsize{\dimexpr\f@size pt\relax}%
  \newcommand*\lineheight[1]{\fontsize{\fsize}{#1\fsize}\selectfont}%
  \ifx\svgwidth\undefined%
    \setlength{\unitlength}{211.04851341bp}%
    \ifx\svgscale\undefined%
      \relax%
    \else%
      \setlength{\unitlength}{\unitlength * \real{\svgscale}}%
    \fi%
  \else%
    \setlength{\unitlength}{\svgwidth}%
  \fi%
  \global\let\svgwidth\undefined%
  \global\let\svgscale\undefined%
  \makeatother%
  \begin{picture}(1,0.52515152)%
    \lineheight{1}%
    \setlength\tabcolsep{0pt}%
    \put(0,0){\includegraphics[width=\unitlength,page=1]{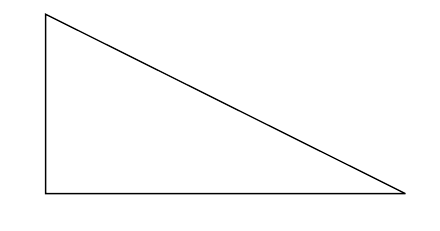}}%
    \put(-0.00268956,0.28243907){\color[rgb]{0,0,0}\makebox(0,0)[lt]{\lineheight{1.25}\smash{\begin{tabular}[t]{l}$F_1$\end{tabular}}}}%
    \put(0.40270489,0.00626405){\color[rgb]{0,0,0}\makebox(0,0)[lt]{\lineheight{1.25}\smash{\begin{tabular}[t]{l}$F_2$\end{tabular}}}}%
    \put(0.44324439,0.34831562){\color[rgb]{0,0,0}\makebox(0,0)[lt]{\lineheight{1.25}\smash{\begin{tabular}[t]{l}$F_3 = \{y=1-\frac{x}{a} \}$\end{tabular}}}}%
    \put(0.07838931,0.02400013){\color[rgb]{0,0,0}\makebox(0,0)[lt]{\lineheight{1.25}\smash{\begin{tabular}[t]{l}$0$\end{tabular}}}}%
    \put(0.89677944,0.00626405){\color[rgb]{0,0,0}\makebox(0,0)[lt]{\lineheight{1.25}\smash{\begin{tabular}[t]{l}$a$\end{tabular}}}}%
    \put(0.01251272,0.4978048){\color[rgb]{0,0,0}\makebox(0,0)[lt]{\lineheight{1.25}\smash{\begin{tabular}[t]{l}$1$\end{tabular}}}}%
  \end{picture}%
\endgroup%
}
\caption{\label{F:acute}  Setup for right triangles}
\hfill
\end{figure}
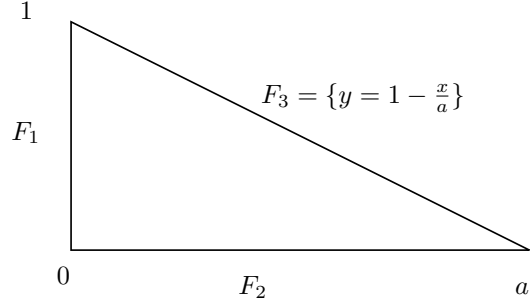

We further assume that $\Omega$ is rational, meaning that all angles
are rational multiples of $\pi$.  
A result of Marklof-Rudnick \cite{Marklof-Rudnick} shows that in this case, there is a
density one (in the sense of \eqref{E:density-1}) subsequence of eigenfunctions $u_{j_k}$  which equidistribute in space
\[
\int_{U} | u_{j_k} |^2 dV \to \frac{\Area (U)}{\Area (\Omega)}
\]
as $h_{j_k} \to 0$.

We will work with this subsequence, and prove the following Theorem:
\begin{theorem}
\label{T:2}
Suppose $\Omega$ is the right triangle oriented as in Figure \ref{F:acute}.  Let $\{ u_j\}$ be the sequence of orthonormal Dirichlet eigenfunctions on $\Omega$.  There exists a density one subsequence $\{u_{j_k}\}$ such that
\[
\limsup_{k\to\infty} \int_\Omega | h \p_x u_{j_k}|^2 dV \leq \frac{7}{8}.
\]
\end{theorem}

\begin{remark}
We again emphasize that this estimate is highly dependent on the orientation of the triangle.  The same proof works for the $y$-derivatives, so that, given $\epsilon>0$, there exists $K$ such that $k \geq K$ implies
\[
\frac{1}{8}-\epsilon \leq \int_\Omega | h \p_x u_{j_k} |^2 dV \leq \frac{7}{8} + \epsilon.
\]
A rotation of the triangle into different coordinates $(s,t)$  is
\[
h \p_x = \alpha h \p_s + \beta h \p_t , \,\,\, h \p_y = -\beta h \p_s + \alpha h \p_t,
\]
where $\alpha^2 + \beta^2 = 1$.  Plugging in to our estimate gives
\[
\frac{1}{8}-\epsilon \leq \int_\Omega | (\alpha h \p_s + \beta h \p_t )u_{j_k} |^2 dV \leq \frac{7}{8} + \epsilon
\]
and similarly
\[
\frac{1}{8}-\epsilon \leq \int_\Omega | (  -\beta h \p_s + \alpha h \p_t )u_{j_k} |^2 dV \leq \frac{7}{8} + \epsilon.
\]
Expanding these quantities does not give us much information unless one of $\alpha$ or $\beta$ is close to zero.

\end{remark}

\begin{remark}
The main idea of the proof is to 
estimate the mass of $h \p_x u$ in strips with that of $u$ in strips.  We then use this
and the results from \cite{Chr-tri} on Neumann data on a whole side to get weak estimates on partial Neumann data.

\end{remark}


\subsection{Quantum Ergodicity}
Roughly speaking, quantum ergodicity (QE) for planar domains states that if the classical billiard flow is ergodic, then there is a density one subsequence of eigenfunctions which equidistribute in phase space \cite{Shn,Zel1,CdV,ZZ}.  That is, this subsequence of eigenfunctions distributes evenly both on the domain and in frequency.  The work of Lindenstrauss \cite{Lindenstrauss} shows that quantum ergodicity can hold for the whole sequence of eigenfunctions, called quantum unique ergodicity (QUE).  The work of Hassell \cite{Hassell-stadium} shows that QUE can fail, so the question of QUE versus non-QUE is very subtle.

In related work, Hassell-Zelditch \cite{HaZe} show that the boundary Neumann data of Dirichlet (and Dirichlet data of Neumann) eigenfunctions satisfy a natural quantum ergodic property, called quantum ergodicity of restrictions (QER).  Work of Toth-Zeldtich \cite{ToZe-1,ToZe-2} extends these results to interior hypersurfaces, again along a density one subsequence.  The work of the first author and Toth-Zelditch \cite{CTZ} proves that QUE implies quantum unique ergodicity for restrictions (QUER) to interior hypersurfaces, at the expense of needing both the (weighted) Dirichlet and Neumann data for the equidistribution.

In \cite{Chr-tri} (see also \cite{Chr-simp} in higher dimensions), the first author proves that for {\it any} planar triangle, the Neumann data of Dirichlet eigenfunctions satisfies an equidistribution identity on each side:

\begin{theorem}[\cite{Chr-tri,Chr-simp}]
Let T be a planar triangle with sides A, B, C of length a, b, c respectively. Consider the (semi-classical) Dirichlet eigenfunction problem \eqref{E:efcn1}
and assume the eigenfunctions are normalized ($||u||^2_{L^2(T)} = 1$).

Then the (semi-classical) Neumann data on the boundary satisfies 
\begin{align}
    \int_A |h\partial_\nu u|^2 dS = \frac{a}{Area(T)}\\
    \int_B |h\partial_\nu u|^2 dS = \frac{b}{Area(T)}\\
    \int_C |h\partial_\nu u|^2 dS = \frac{c}{Area(T)}
\end{align}
where  $h\partial_\nu$ is the semi-classical normal derivative on $\partial T$, $dS$ is the arc-length measure, and $Area(T)$ is the area of the triangle T.
\end{theorem}

%

\begin{remark}
This property is called `equidistribution' as the Neumann data on each side is proportional to the length of that side, and the quantities are exactly what would be predicted if QUER was satisfied on the boundary.  However, we stress that the integrals need to be over the {\it whole} side.  Distribution of Neumann data over subsets of the sides is the topic of this paper, and indeed Theorem \ref{T:1} shows this fails in the simplest possible case of a right isosceles triangle.

%

\end{remark}

There are several natural questions that arise based on this result. What can be said about the Neumann data on subsets of sides? Can we get an analogous result for subsets, even if we only consider results in a high energy limit or subsequences of a specific density? What about volume integrals over the same domain?

To answer these questions in Euclidean space, we will begin by dealing with the case of a right isosceles triangle as we have explicit solutions to work with. We will then move on to numerical results which will allow us to get data from triangles to properly set expectations for these tough analytical problems. 

\subsection{Immediate Questions}
Based on this result, this paper is concerned with two immediate questions.

\begin{question}
Is it true that 
\begin{equation}
    \forall \omega\subset \partial T, \,\, \lim_{h\rightarrow 0} \int_\omega|h\partial_\nu u|^2dS \rightarrow \frac{m(\omega)}{Area(T)},
\end{equation} 
where $m(\omega)$ is the measure of the set $\omega$?
 \end{question}
This is just an extension of the equidistribution result to arbitrary subsets. A second obvious question would be the following: 
\begin{question}
Is it true that 
\begin{equation}
    \forall h>0, \,\,\int_T |h\partial_y u|^2dV = \int_T |h\partial_x u|^2dV = \frac{1}{2}?
\end{equation}
\end{question}

\subsection{Connecting boundary integrals to interior energy}
Let us continue to work with the right triangle given by Figure \ref{F:acute}.  We duplicate the argument from \cite{Chr-tri} but with the vector field $X = x \p_x$.  The point is that $X = 0$ on $\{ x = 0 \}$ and $X$ is tangential on $\{ y = 0 \}$.  Along the side $F_3 = \{ 0 \leq x \leq a: \, y = 1 - \frac{x}{a} \}$ we have the tangent derivative is $\p_\tau = \gamma^{-1} ( a \p_x - \p_y)$ where $\gamma = (1 + a^2)^{1/2}$.  The normal derivative is then $\p_\nu  = \gamma^{-1} (\p_x + a \p_y )$.  Since $u = 0$ along $F_3$, we have $\p_\tau u = \gamma^{-1} (a \p_x - \p_y ) u = 0$, or $\p_y u = a \p_x u$ on $F_3$.  Hence
\[
\p_\nu u = \gamma^{-1} (\p_x + a \p_y ) u = \gamma^{-1} (1 + a^2) \p_x u = \gamma \p_x u,
\]
so that
\[
\p_x u = \gamma^{-1} \p_\nu u
\]
 along $F_3$.

Then using the same integrations by parts as in \cite{Chr-tri}, we have
\begin{equation}
\label{E:dx-comm-1001}
\int_\Omega ([-h^2 \Delta -1, X] u) \bu dV = 2 \int_\Omega (-h^2 \p_x^2 u ) \bu dV = 2 \int_\Omega | h \p_x u |^2 dV.
\end{equation}
On the other hand, unpacking the commutator and applying Green's formula just like in \cite{Chr-tri}, we have
\begin{align}
\int_\Omega & ([-h^2 \Delta -1, X] u) \bu dV \notag \\
& = \int_{\p \Omega} (hX u) h \p_\nu \bu dS \notag \\
& = \int_{F_3} x (h \p_x u) h \p_\nu \bu dS \notag \\
& = \gamma^{-1}\int_{F_3}x | h \p_\nu u|^2 dS.
\end{align}
This shows that, {\it if} we knew that the Neumann data along $F_3$ was quantum uniquely ergodic on subsets of the side, we would have
\[
 \gamma^{-1}\int_{F_3}x | h \p_\nu u|^2 dS = \frac{1}{2} \gamma^{-1} \int_{F_3} | h \p_\nu u|^2 dS + o(1).
 \]
 Then from \cite{Chr-tri} we know the integral on the right is equal to $1$.  Rearranging and equating to \eqref{E:dx-comm-1001}, this computation {\it would} tell us that
 \[
 \int_\Omega |h \p_x u |^2 dV = 1/2 + o(1),
 \]
 however this ``obvious'' conjecture appears to be false.

 Similar computations with vector fields like $X = y \p_x$ connects the quantity $\int_\Omega (h \p_x u )(h \p_y \bu) dV$ to other weighted boundary integrals,
 so weighted boundary integrals are essential to understanding interior energy distribution of eigenfunctions.

\section{Analytical Results for Right Isosceles Triangle}
\subsection{Introducing the eigenfunctions on the Right Isosceles Triangle}
As we have explicit formulas for eigenfunctions of the Laplacian on a right isosceles triangle, we will study these functions both to prove conclusively some results and as a baseline for results we discuss later on almost isosceles triangles. For this section, $T$ will be a triangle in Euclidean space with vertices $(0,0), (1,0),$ and $(0,1)$. For the rest of this paper, we will deal with triangles with vertices at the origin and at $(0,1)$. We will identify triangles by the x coordinate of the third vertex, which will always be on the positive x-axis.

\begin{theorem}
Let $T$ be as previously described. Then the following formula exhausts all of the eigenfunctions of the Laplacian on $T$ that satisfy Dirichlet boundary conditions with $m,n \in \mathbb{Z}, m\neq n$.

\begin{equation}
    u_{m n} = c_{mn}\sin(n\pi x)\sin(m\pi y) + d_{mn}\sin(m\pi x)\sin(n\pi y).
\end{equation}

With the additional constraint that $c_{m n} = d_{m n}$ if $m$ and $n$ are of opposite parity and $c_{m n} = -d_{m n}$ if $m$ and $n$ have the same parity. Additionally, by normalization, $c^2_{m n} = d^2_{m n}=4$.
\end{theorem}

\begin{proof}
We will show that these functions are exhaustive, satisfy the boundary conditions, and satisfy the eigenfunction equation. We achieve this expression by noticing that reflecting $T$ across the line $y= 1 - x$ gives a square. The eigenfunctions of the Laplacian on a square are well known, so we know immediately that this list is exhaustive.   We then just have to check all of the usual requirements to verify these are indeed eigenfunctions on the isosceles triangle. 

Clearly $x = 0 \implies u_{mn} = 0$ and  $y = 0 \implies u_{mn} = 0$. Checking $y=1-x$ gives the following expression:

\begin{align*}
    u_{m n}(x,1-x) &= c_{mn}\sin(n\pi x)\sin(m\pi - m\pi x) + d_{mn}\sin(m\pi x)\sin(n\pi - n\pi x) \\
    &= (-1)^{m+1}c_{mn}\sin(n\pi x)\sin(m\pi x) + (-1)^{n+1}d_{mn}\sin(m\pi x)\sin(n\pi x)
\end{align*}

That $u_{m n}$ solves the eigenfunction equation carries over from the fact that these are restricted eigenfunctions of the square. A simple computation gives the eigenvalue as $h^{-2} = \pi^2(n^2+m^2)$ which is the same as in the square case.
\end{proof}

\subsection{Calculating the Volume Integral for the Right Isosceles Triangle}

One of the metrics we are interested in is $\int_{T}|h\partial_y u_{m n}|^2 dV$. We will refer to this as the ``$y$ volume integral'' for expository convenience.  The derivative integrals and the function integrals are related by the equation $\int_{T}|h\partial_x u_{mn}|^2 + |h\partial_y u_{mn}|^2 dV = \int_{T}|u_{mn}|^2 dV = 1$. This expression can be achieved by simple integration by parts, as we have
\begin{equation}
    h^2\int_{T}\partial_x u_{mn}\partial_x u_{mn} + \partial_y u_{mn}\partial_y u_{mn} dV  
    = \int_T u_{mn}(-h^2 \Delta u_{mn})dV = \int_Tu_{mn}^2dV,
\end{equation}
where the boundary terms are zero as we assume Dirichlet boundary conditions, and the last substitution uses $u_{mn}$ being an eigenfunction.

Quantum ergodicity can be interpreted as most of the eigenfunctions tending towards equidistribution, and a consequence of this is the volume integrals of both derivatives tending to $\frac{1}{2}$. In the simple case of the right isosceles triangle, we have equality in $L^2$ norms of $\partial_xu$ and $\partial_yu$ by symmetry, so we in fact have equality for every eigenvalue. As the volume metrics are completely understood in this case, it is natural to investigate the analogous metrics on the boundary as well.

Later on, the $y$ volume integral will be an important metric throughout this paper as a way to test quantum ergodicity. We can calculate an integral over the entire domain to test if our functions are quantum ergodic compliant, which is far easier numerically than dealing with subsets of the domain.

\subsection{Showing Equidistribution fails for subsets of the boundary of the Right Isosceles Triangle}

To begin addressing the question of what happens on subsets of sides, we will explore the amount of the Neumann data on one half of the side on the $x$-axis compared to the other. As the sum of the data on both halves of the bottom side is constant, we only consider the bottom left face. This calculation is the same for the left face, as we have $xy$ symmetry. As such we will define:

\begin{equation}
    I_l(m,n) = \frac{1}{2}\int_0^{1/2}|h\partial_\nu u_{mn}(x,0)|^2dx
\end{equation}
\begin{equation}
    I_r(m,n) = \frac{1}{2}\int_{1/2}^{1}|h\partial_\nu u_{mn}(x,0)|^2dx.
\end{equation}

We always have $I_l(m,n) + I_r(m,n) = 1$ by the result for Neumann data on the entire side in \cite{Chr-tri}. $I_l = I_r = \frac{1}{2}$ represents equidistribution on the two halves. This would not be enough to say the Neumann data is uniformly distributed, but we will see almost immediately that equidistribution does not hold even in this simple case. 

\begin{theorem}
There exists $m,n$ such that $I_l(m,n)\neq \frac{1}{2}$. Moreover, the subsequence $I_l(k,k+1)$ converges to a value other than $1/2$. 
\end{theorem}

\begin{proof}
On the bottom side the normal derivative is just $-\partial_y$. By direct calculation:
\begin{align*}
    \Bigl|h\partial_\nu u|_{y=0}\Bigr|^2 &  = h^2\Bigl(c^2_{mn}m^2\pi^2\sin^2(n\pi x) + 2c_{mn}d_{mn}\pi^2 n m\sin(n\pi x)\sin(m\pi x) \\
    & \quad +  d^2_{mn}n^2\pi^2\sin^2(m\pi x)\Bigr),
\end{align*}
which has the following anti-derivative $F$ ($m \neq n$) using basic trig identities:
\begin{align*}
    F(x)& = h^2c^2_{mn}m^2\pi^2\Bigl(\frac{1}{2}x - \frac{1}{4n\pi}\sin(2n\pi x)\Bigr)\\
    & +h^2c_{mn}d_{mn}\pi^2 n m\Bigl(\frac{1}{\pi(n-m)}\sin(\pi(n - m)x)-\frac{1}{\pi(n+m)}\sin(\pi(n + m)x)\Bigr)\\
    &+  h^2d^2_{mn}n^2\pi^2\Bigl(\frac{1}{2}x - \frac{1}{4m\pi}\sin(2m\pi x)\Bigr) + C
\end{align*}
\
This lets us calculate explicitly:

\begin{align*}
    &F(0) &&= 0, \\
    &F(1/2) &&= \frac{h^2\pi^2(c^2_{mn}m^2  + d^2_{mn}n^2)}{4}+ h^2c_{mn}d_{mn}\pi n m\Bigl(\frac{\sin(\frac{1}{2}\pi(n - m)}{(n-m)}-\frac{\sin(\frac{1}{2}\pi(n + m))}{(n+m)}\Bigr)\\
    & &&=  1 + h^2c_{mn}d_{mn}\pi n m\Bigl(\frac{\sin(\frac{1}{2}\pi(n - m)}{(n-m)}-\frac{\sin(\frac{1}{2}\pi(n + m))}{(n+m)}\Bigr), \text{ and}\\
    &F(1) &&= \frac{h^2}{2}(c^2_{mn}m^2\pi^2 + d^2_{mn}n^2\pi^2) = 2.
\end{align*}

This gives us the following:
\begin{align*}
    2I_l(m,n) & = F(1/2) - F(0) \\
    & = 1 + h^2c_{mn}d_{mn}\pi n m\Bigl(\frac{\sin(\frac{1}{2}\pi(n - m))}{(n-m)}-\frac{\sin(\frac{1}{2}\pi(n + m))}{(n+m)}\Bigr).
\end{align*}

Note that if $m+n$ is even, then $2I_l(m,n) = 1$. We will then consider situations where $m+n$ is odd, which forces $c_{mn}d_{mn} = 4$ as $m+n$ is odd when $m$ and $n$ have different parities. Furthermore, as we push $m$ and $n$ to infinity, the term multiplied by $\frac{1}{n+m}$ will go to 0 as $h^2 = (\pi^2n^2+\pi^2m^2)^{-1}$. We will numerically show what all of these values are later on, but to construct our subsequence consider $m_k = k$ and $n_k = k+1$. This is a subsequence for which the terms multiplied by $\frac{1}{n-m}$ will have the largest magnitude. Restricting to this subsequence and plugging in exact values for $h^2c_{mn}d_{mn}$ gives us:

\begin{align*}
    2I_1(m_k,n_k) &= 1 + 4(\pi^2(2k^2+2k+1))^{-1}\pi(k^2+k)\Bigl(\sin(\frac{\pi}{2}) - \frac{\sin(\frac{\pi}{2}(2k+1))}{2k+1}\Bigr)\\
    &\sim 1 + \frac{2}{\pi} + \mathcal{O}(k^{-1})\\
\end{align*}

This implies that, for this subsequence of proportions $I_l(k,k+1)$, we have that $I_l(k,k+1) \rightarrow \frac{1}{2} + \frac{1}{\pi} \approx .8183$. It is then immediately the case that $I_r(k,k+1)\rightarrow \approx .1817$. This is our subsequence that does not equidistribute on subsets in the limit.
\end{proof}

The lack of equidistribution on subsets of the sides, even in the limit, is more surprising than the $y$ volume integral result. This contradicts the original conjecture by the first author that there was a uniform distribution in the limit. In this case the long term behavior of these proportions can be completely described. The following computation is identical to the previous one but done in generality.

\begin{corollary}
Let $m$ and $n$ be integers such that $n-m = j$ where $j$ is an odd integer. Then we have an explicit formula for $I_l(m,n).$

\end{corollary}
\begin{proof}
We proceed in the same manner as the previous proof. By plugging in our assumed values we have:
\begin{align*}
    2I_l(m, m+j) &= 1 + 4\pi^{-1}(2m^2+2mj+j^2)^{-1}(m^2+mj)(\frac{\sin(\frac{\pi}{2}j)}{j} - \frac{\sin(\frac{\pi}{2}(2m+j))}{2m+j})\\
    &\sim 1 + \delta(j)\frac{2}{j\pi} + \mathcal{O}(m^{-1})\\
\end{align*}
and therefore 
\begin{align*}
    I_l(m,n) \sim \frac{1}{2}(1 + \delta(j)\frac{2}{j\pi} + \mathcal{O}(k^{-1}))\\
    I_r(m,n) \sim \frac{1}{2}(1 - \delta(j)\frac{2}{j\pi} + \mathcal{O}(k^{-1})),
\end{align*}
where $\delta(j) = 1$ if $j \equiv 1$ (mod 4) and $\delta(j) = -1$ if $j \equiv 3$ (mod 4).
\end{proof}
This computation also shows that, in the limit, the running average of these two values will both be 1/2. The subsequence of $m,n$ such that they are separated by a fixed integer is density 0 in the sequence of $m,n$. We can also see that the limit of these subsequences, $I_l(m,m+j), I_r(m,m+j)$ goes to $1/2$ when we take the separation integer $k$ to infinity. This ensures via a straightforward limit argument that the running average of each piece also goes to $1/2$.

The reason for this can clearly be seen in the explicit computations, as $m, n$ values that are close together produce disturbances whose magnitude is not changed when $m$ and $n$ are pushed to infinity so long as that separation is maintained. However, encountering $m$ and $n$ pairs with that separation becomes less and less likely as $m$ and $n$ increase. All of this has been numerically confirmed.

\begin{figure}[h]
\begin{subfigure}{0.5\textwidth}
\includegraphics[width=1\linewidth, height=5cm]{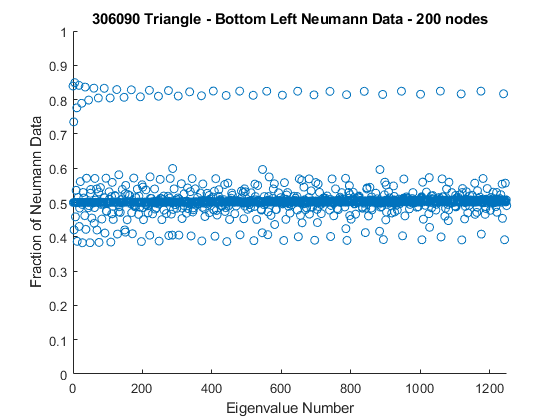} 
\caption{Computed Bottom Left Neumann Data for the Right Isosceles}
\label{fig:subim1}
\end{subfigure}
\begin{subfigure}{0.5\textwidth}
\includegraphics[width=1\linewidth, height=5cm]{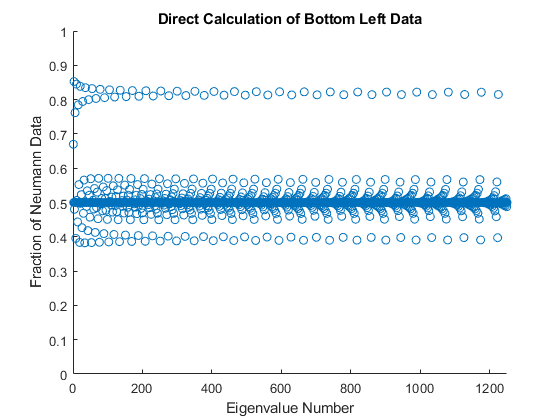}
\caption{Plot of Bottom Left Neumann Data using Derived Formula}
\label{fig:subim2}
\end{subfigure}
\caption{Bottom Left Neumann Data Plots: Computed and Analytical}
\label{fig:image2}
\end{figure}

These two plots are not exactly the same as the direct calculation orders points differently. Moreover, the accuracy of the boundary integrals, especially because we are dividing them, is not enough to perfectly align these graphs.

This result establishes that equidistribution fails even on simple subsets of simple triangles. In this next section we will expand this result to state these proportions need not even be bounded.

\newpage

\section{Proof of Theorem \ref{T:2}}

In this section, we use the result of Marklof-Rudnick \cite{Marklof-Rudnick} to prove Theorem \ref{T:2}. The idea is to compare the integrals of $|h \p_x u|^2$ to those of $|u|^2$ in strips in the triangle, and then use the results from \cite{Chr-tri} to compare the integrals of $|h \p_x u|^2$ to boundary integrals of Neumann data.
\begin{proof}
We drop the subscript and subsequence notation and simply write $u$ for our density one subsequence.

On side $F_1$, the normal derivative is $\p_\nu = -\p_x$, on $F_2$ the
normal derivative is $\p_\nu = - \p_y$, and on $F_3$, the tangent
derivative is
$\p_\tau = \gamma^{-1} ( a\p_x - \p_y)$ and the normal derivative is 
$\p_\nu = \gamma^{-1} (  \p_x + a\p_y )$.  Here $\gamma = (1 +
a^2)^\half$ is the normalizing constant.  That means that on $F_1$, $\p_y u
= 0$, on $F_2$, $\p_x u = 0$, and on $F_3$, $\p_x u = \frac{1}{\gamma}
\p_\nu u$ and $\p_y u = \frac{a}{\gamma} \p_\nu$ as usual.

Fix $0 < \beta < a$ and $\delta>0$  independent of $h$, with $\delta$ sufficiently small so that $0 < \beta - \delta^2 < \beta + \delta < \beta + \delta + \delta^2 < a$.  Let $\chi(x)$ be a smooth function satisfying 
\begin{itemize}
\item $\chi(x) \equiv 0 \text{ for } 0 \leq x \leq \beta - \delta^2$,

\item $\chi(x) \equiv 1 \text{ for } \beta + \delta + \delta^2 \leq x \leq a$,

\item $\chi'(x) \geq 0$,

\item $\chi' = \frac{1}{\delta} + \O ( \delta) \text{ for } \beta \leq x \leq \beta + \delta$

\item $|\chi^{(k)} | = \O (\delta^{-k})$.
\end{itemize}
See Figure \ref{F:chi} for a sketch of such a function.

    \begin{figure}
\hfill
\centerline{
\begingroup%
  \makeatletter%
  \providecommand\color[2][]{%
    \errmessage{(Inkscape) Color is used for the text in Inkscape, but the package 'color.sty' is not loaded}%
    \renewcommand\color[2][]{}%
  }%
  \providecommand\transparent[1]{%
    \errmessage{(Inkscape) Transparency is used (non-zero) for the text in Inkscape, but the package 'transparent.sty' is not loaded}%
    \renewcommand\transparent[1]{}%
  }%
  \providecommand\rotatebox[2]{#2}%
  \newcommand*\fsize{\dimexpr\f@size pt\relax}%
  \newcommand*\lineheight[1]{\fontsize{\fsize}{#1\fsize}\selectfont}%
  \ifx\svgwidth\undefined%
    \setlength{\unitlength}{324.86925888bp}%
    \ifx\svgscale\undefined%
      \relax%
    \else%
      \setlength{\unitlength}{\unitlength * \real{\svgscale}}%
    \fi%
  \else%
    \setlength{\unitlength}{\svgwidth}%
  \fi%
  \global\let\svgwidth\undefined%
  \global\let\svgscale\undefined%
  \makeatother%
  \begin{picture}(1,0.28318107)%
    \lineheight{1}%
    \setlength\tabcolsep{0pt}%
    \put(0,0){\includegraphics[width=\unitlength,page=1]{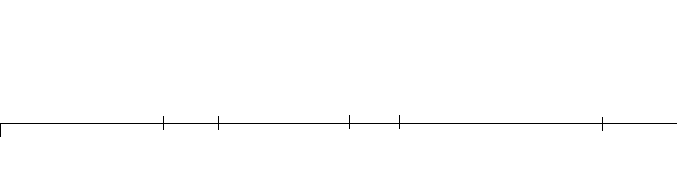}}%
    \put(0.00408913,0.0392053){\makebox(0,0)[lt]{\lineheight{1.25}\smash{\begin{tabular}[t]{l}$0$\end{tabular}}}}%
    \put(0.17580627,0.04744765){\makebox(0,0)[lt]{\lineheight{1.25}\smash{\begin{tabular}[t]{l}$\beta-\delta^2$\end{tabular}}}}%
    \put(0.31649805,0.04964561){\makebox(0,0)[lt]{\lineheight{1.25}\smash{\begin{tabular}[t]{l}$\beta$\end{tabular}}}}%
    \put(0.46312859,0.05057559){\makebox(0,0)[lt]{\lineheight{1.25}\smash{\begin{tabular}[t]{l}$\beta + \delta$\end{tabular}}}}%
    \put(0.59323207,0.05135761){\makebox(0,0)[lt]{\lineheight{1.25}\smash{\begin{tabular}[t]{l}$\beta + \delta + \delta^2$\end{tabular}}}}%
    \put(0.8876344,0.05309071){\makebox(0,0)[lt]{\lineheight{1.25}\smash{\begin{tabular}[t]{l}$a$\end{tabular}}}}%
    \put(0,0){\includegraphics[width=\unitlength,page=2]{chi.pdf}}%
  \end{picture}%
\endgroup%
}
\caption{\label{F:chi}  The function $\chi$.}
\hfill
\end{figure}

Let $X = \chi(x) \p_x$, and run the
usual Rellich commutator argument as in \cite{Chr-tri}:
\[
\int ([ -h^2 \Delta -1 , X] u) \bu dV = -2 \int (\chi' h^2 \p_x^2 u)
\bu dV + \O(h) = 2 \int \chi' | h \p_x u |^2 dV + \O(h).
\]
Let $\Omega_\beta =  \Omega \cap \{ \beta -\delta^2 \leq x \leq \beta +
\delta + \delta^2 \}$ so that $\supp \chi' \subset \Omega_\beta$.  Further let $\tOmega_\beta = \Omega \cap \{ \beta \leq x \leq \beta + \delta \}$ so that $\chi' = \delta^{-1} + \O(\delta)$ on $\tOmega_\beta$.

We write
\begin{align}
2 \int_\Omega \chi' | h \p_x u |^2 dV & = 2 \int _{\Omega_\beta} \chi' | h \p_x u |^2 dV \notag \\
& \leq  2  \int_{\Omega_\beta}  \chi' (|h \p_x u |^2 + | h \p_y u |^2 ) dV  \notag \\
& = 2\int_{\Omega_\beta} \chi' ( -h^2 \Delta u ) \bu dV 
- 2\int_{\Omega_\beta}(h \nabla \chi') \cdot (h \nabla u) \bu dV \\
& = 2\int_{\Omega_\beta} \chi' ( -h^2 \Delta u ) \bu dV 
+ \O(h \delta^{-2}) \notag \\
& = 2\int_{\Omega_\beta}  \chi' |        u |^2  dV  + \O(h \delta^{-2}) \notag \\
& \leq 2 (\delta^{-1}+ \O(\delta)) \int_{\Omega_\beta} | u |^2 dV + \O(h \delta^{-2}).  \notag 
\end{align}

 We have
\[
\Area ( \Omega_\beta) = \left[\frac{ 1 - \frac{(\beta-\delta^2)}{a} + 1 - \frac{(\beta + \delta + \delta^2)}{a}}{2} \right] ( \delta + 2 \delta^2) = (1 - \frac{\beta}{a} )\delta + \O(\delta^2).
\]
Hence
\[
\frac{ \Area ( \Omega_\beta)}{\Area(\Omega)} = \frac{(1 - \frac{\beta}{a}) \delta}{a/2} + \O(\delta^2).
\]
Then the result of Marklof-Rudnick \cite{Marklof-Rudnick} implies
\[
2 (\delta^{-1}+ \O(\delta)) \int_{\Omega_\beta} | u |^2 dV = 4 \frac{(1 - \frac{\beta}{a}) }{a} + \O(\delta) + o(1),
\]
so that
\begin{equation}
\int ([ -h^2 \Delta -1 , X] u) \bu dV \leq  4 \frac{(1 - \frac{\beta}{a}) }{a} + \O(\delta) + \O(h \delta^{-2}) +  o(1).
\label{E:brutal-est}
\end{equation}

On the other hand,
\[
\int ([ -h^2 \Delta -1 , X] u) \bu dV = \int_{\p \Omega} \chi(x) (h \p_x u ) h \p_\nu \bu dS.
\]
On $F_1$, $\chi(0) = 0$ and on $F_2$, $\p_x$ is tangential so $\p_x u = 0$.  On $F_3$, $\p_x u = \gamma^{-1} \p_\nu u$, so that
\[
\int ([ -h^2 \Delta -1 , X] u) \bu dV = \gamma^{-1} \int_{F_3} \chi(x) | h \p_\nu u|^2 dS.
\]
Putting this together, 
\begin{equation}
\label{E:bdy-upper}
\gamma^{-1} \int_{F_3} \chi(x) | h \p_\nu u|^2 dS \leq \frac{4}{a} ( 1 - \frac{\beta}{a}) + \O(\delta) + o(1).
\end{equation}
We will use \eqref{E:bdy-upper} to estimate the Neumann data on part of $F_3$.  Since $\chi \equiv 1$ on $\{ \beta + \delta + \delta^2 \leq x \leq a \}$, we have
\begin{equation}
\label{E:bdy-upper-1}
\gamma^{-1} \int_{F_3 \cap \{ \beta + \delta + \delta^2 \leq x \leq a \}} | h \p_\nu u|^2 dS \leq \gamma^{-1} \int_{F_3} \chi(x) | h \p_\nu u|^2 dS \leq \frac{4}{a} ( 1 - \frac{\beta}{a}) + \O(\delta) + o(1).
\end{equation}

We now use another commutator type argument to compare the mass of $h \p_x u$ on the whole triangle to the Neumann data on part of the boundary.  To that end, let $X = (1-x/a) \p_x$.  Then $[-h^2 \Delta -1, X ] = 2a^{-1} h^2 \p_x^2$ so that
\[
\int_\Omega ([-h^2 \Delta -1 , X] u) \bu dV = \frac{2}{a} \int_{\Omega} (h^2 \p_x^2 u) \bu dV = -\frac{2}{a} \int_\Omega | h \p_x u |^2 dV.
\]
On the other hand,
\[
\int_\Omega ([ -h^2 \Delta -1 , X ] u) \bu dV = \int_{\p \Omega} (1-x/a) h \p_x u h \p_\nu \bu dS.
\]
On $F_1$, $x = 0$ so $X = \p_x = -\p_\nu$.  On $F_2$, $\p_x$ is tangential, so that $X u = 0$ on $F_2$.  On $F_3$, we have $\p_x u = \gamma^{-1} \p_\nu u$ as before.  That means
\[
 \int_{\p \Omega} (1-x/a) h \p_x u h \p_\nu \bu dS = -\int_{F_1} | h \p_\nu u|^2 dS + \gamma^{-1} \int_{F_3} (1 - x/a) | h\p_\nu u |^2 dS.
 \]
 From \cite{Chr-tri}, we know $\int_{F_1} | h \p_\nu u|^2 dS = \frac{2}{a}$, so that
 \[
  \int_{\p \Omega} (1-x/a) h \p_x u h \p_\nu \bu dS = -\frac{2}{a} + \gamma^{-1} \int_{F_3}(1 - x/a) | h \p_\nu u|^2 dS.
  \]
 Rearranging, we have
 \begin{equation}
 \label{E:hpx-F3}
   \frac{2}{a} \int_\Omega | h \p_x u |^2 dV = \frac{2}{a} - \gamma^{-1} \int_{F_3}(1 - x/a) | h \p_\nu u|^2 dS.
   \end{equation}

   To get an upper bound on the left hand side, we need a lower bound on the integral
   \[
   \gamma^{-1} \int_{F_3} (1 - x/a) | h \p_\nu u|^2 dS,
 \]
which we do by comparing to the part of the boundary isolated by our cutoff function $\chi$.  $\chi(x) \equiv 1$ for $x \geq \beta + \delta + \delta^2$, and we have an upper bound on the boundary data in this range, not a lower bound.   We write
\begin{align}
\gamma^{-1} & \int_{F_3} (1 - x/a) | h \p_\nu u |^2 dS  \label{E:bdy-2}\\
& = 
  \gamma^{-1} \int_{F_3 \cap \{ x \geq \beta + \delta + \delta^2 \} } (1 - x/a) | h \p_\nu u |^2 dS  \notag \\
  & \quad + 
  \gamma^{-1} \int_{F_3 \cap \{ x \leq \beta + \delta + \delta^2 \}} (1 - x/a) | h \p_\nu u |^2 dS  \notag \\
&   \geq (1 - (\beta + \delta + \delta^2)/a)\gamma^{-1} \int_{F_3 \cap \{ x \leq \beta + \delta + \delta^2 \}} | h \p_\nu u |^2 dS. \notag
\end{align}
We have
\begin{align*}
\gamma^{-1} & \int_{F_3 \cap \{ x \leq \beta + \delta + \delta^2 \}} | h \p_\nu u |^2 dS \\
& = \gamma^{-1} \int_{F_3} | h \p_\nu u |^2 dS - \gamma^{-1} \int_{F_3 \cap \{ x \geq \beta + \delta + \delta^2 \}} | h \p_\nu u|^2
\end{align*}
and now our upper bound \eqref{E:bdy-upper-1} in the region $x \geq \beta + \delta + \delta^2$ is useful.  Again using the main result from \cite{Chr-tri}, we have 
\[
\gamma^{-1} \int_{F_3} | h \p_\nu u |^2 dS = \frac{2}{a},
\]
so 
\begin{align*}
\gamma^{-1} & \int_{F_3 \cap \{ x \leq \beta + \delta + \delta^2 \}} | h \p_\nu u |^2 dS \\
& = \gamma^{-1} \int_{F_3}  | h \p_\nu u |^2 dS - \gamma^{-1} \int_{F_3 \cap \{ x \geq \beta + \delta + \delta^2 \}} | h \p_\nu u|^2 \\
& \geq \frac{2}{a} - \frac{4}{a} ( 1 - \frac{\beta}{a}) + \O(\delta) + o(1).
\end{align*}

Plugging into \eqref{E:bdy-2}, we have
\begin{align*}
\gamma^{-1} & \int_{F_3} (1 - x/a) | h \p_\nu u |^2 dS \\
& \geq (1 - (\beta + \delta + \delta^2)/a)\gamma^{-1} \int_{F_3 \cap \{ x \leq \beta + \delta + \delta^2 \}} | h \p_\nu u |^2 dS \\
& \geq (1 - (\beta + \delta + \delta^2)/a) \left(\frac{2}{a} - \frac{4}{a} ( 1 - \frac{\beta}{a}) + \O(\delta) + o(1) \right).
\end{align*}

Combining with \eqref{E:hpx-F3}, we have
\begin{align*}
\frac{2}{a} \int_\Omega | h \p_x u |^2 dV & = \frac{2}{a} - \gamma^{-1} \int_{F_3}(1 - x/a) | h \p_\nu u|^2 dS \\
& \leq \frac{2}{a} - (1 - (\beta + \delta + \delta^2)/a)
(\frac{2}{a} - \frac{4}{a} ( 1 - \frac{\beta}{a}) + \O(\delta) + o(1)) 
\end{align*}
and rearranging,
\begin{align}
\int_\Omega | h \p_x u |^2 dV & \leq 1-(1 - (\beta + \delta + \delta^2)/a)(1-2( 1 - \frac{\beta}{a}) + \O(\delta) + o(1) \notag \\
& = 1 - (1 - \beta/a)(1-2(1-\beta/a)) + \O(\delta) - o(1). \label{E:last-est}
\end{align}
Optimizing in the variable $(1-\beta/a)$ gives $(1-\beta/a) = 1/4$, or
\[
\int_\Omega | h \p_x u |^2 dV \leq 1-(1/4)(1/2) = 7/8 + \O(\delta) + o(1)
\]
as asserted in the Theorem.

\end{proof}

\begin{remark}
The biggest loss in the proof is from brutally estimating the integral of $| h \p_x u |^2$ in strips by the integral of $| u |^2$, which is clearly a very crude estimate.  It is nevertheless interesting to note that if we knew that the integral of $| h \p_x u |^2$ in strips was {\it half} that of $|u|^2$, which would be predicted by quantum ergodicity, the proof still does not give the expected estimate on the whole triangle.  Indeed, in \eqref{E:brutal-est}, quantum ergodicity would have given 
$ 2\frac{(1 - \frac{\beta}{a}) }{a} + \O(\delta) + o(1)$ instead of $ 4 \frac{(1 - \frac{\beta}{a}) }{a} + \O(\delta) + o(1)$.
 As in the end of the proof, this would give
 \[
 \int_\Omega | h \p_x u |^2 dV \leq
 1 - (1 - \beta/a)(1-(1-\beta/a)) + \O(\delta) - o(1)
\]
in place of \eqref{E:last-est}.  Optimizing again in the variable $(1-\beta/a)$ yields
$(1-\beta/a) = 1/2$, for a bound of $3/4 + \O(\delta) + o(1)$.  So even if we knew more aboud energy distribution compared to distribution, the techniques of proof in this paper give an unsatisfactory answer.

\end{remark}

\begin{remark}
Note this is particular to triangles.  Indeed, if $\Omega = [0 ,
  \pi]^2$, a basis of  eigenfunctions consists of $u_{mn} = c_{mn} \sin (mx)\sin(ny)$,
where $c_{mn} = 2/\pi$ is the appropriate normalization constant.
Let $U \subset \Omega$ be an open set.  We have
\begin{align*}
\int_U | u |^2 dV & = \int_U |c_{mn}|^2 ( 1/2 - 1/2 \cos (2mx))(1/2 -
1/2 \cos(2ny)) dV \\
& = \pi^{-2} \int_U (1 - \cos(2mx) - \cos(2nx) + \cos(2mx) \cos (2ny))
dV \\
& = \frac{\Area(U)}{\Area(\Omega)} + \O(m^{-1} + n^{-1}).
\end{align*}

On the other hand,
\[
\int_\Omega | h \p_x u |^2 dV = \int_\Omega h^2 m^2 (4/\pi^2) | \cos(m
x) \sin (ny)|^2 dV = h^2 m^2,
\]
and similarly $\int_\Omega | h \p_y u |^2 dV = h^2 n^2.$

Suppose we
are interested in $\{ n \geq M m \}$ for large $M$.  Then
$\# \{ m^2 + n^2 \leq R^2 : n \geq Mm \} \sim R^2/M$, so has density
$\sim 1/M >0$, but $\int_\Omega | h \p_x u |^2 dV \leq  M^{-2}$.

This shows that these eigenfunctions with $n \geq M m$ satisfy the spatial
equidistribution as in Marklof-Rudnick:
\[
\int_U | u |^2 dV =  \frac{\Area(U)}{\Area(\Omega)} + \O (Mh)
\]
but do not have the frequency lower bound property $\int_\Omega | h
\p_x u |^2 dV \geq 1/8$.

\end{remark}

\section{An Almost Right Isosceles Triangle}
With the right isosceles studied, it is natural to see what happens when the domain is perturbed slightly. We will investigate the '.99 triangle', or, the triangle with vertices $\{(0,0), (0,1), (.99,0)\}$. Analytical solutions cannot be found, but we can use numerical techniques. Using FreeFEM, an online tool for using the finite element method to solve PDEs, we have calculated the first 1250 eigenfunctions and plotted their relevant data. 

\begin{figure}[h]
\begin{subfigure}{0.5\textwidth}
\includegraphics[width=1\linewidth, height=5cm]{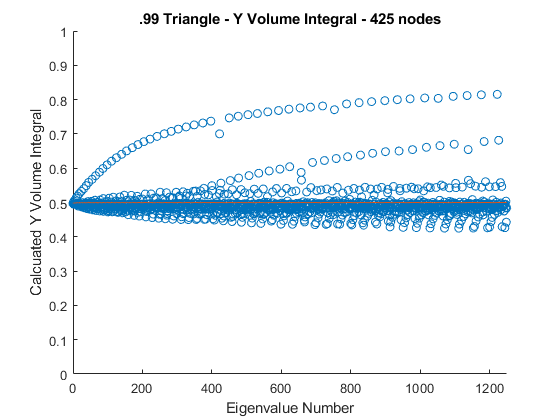} 
\label{fig:subim1}
\end{subfigure}
\begin{subfigure}{0.5\textwidth}
\includegraphics[width=1\linewidth, height=5cm]{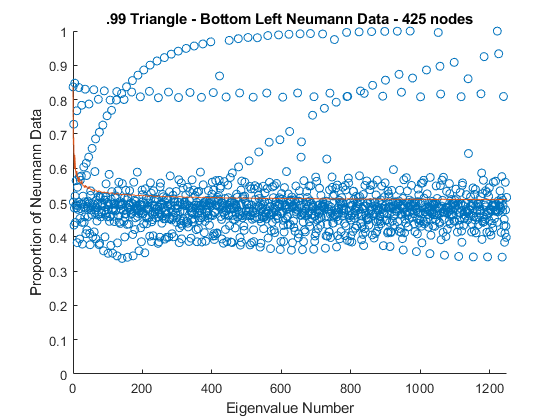}
\label{fig:subim2}
\end{subfigure}
\caption{.99 Triangle - 1250 Eigenvalues - Y Volume Integral and Bottom Left Neumann plots}
\label{fig:image2}
\end{figure}

The $y$ volume integral plot is no longer constant and have noticeable structure. There are at least two branches that can be seen in our range of eigenvalues. These branches correspond to subsequences of eigenfunctions whose $y$ volume integrals seem to not approach $\frac{1}{2}$. There is also a large band with sizable separation from $\frac{1}{2}$. As the energy increases, even the less unusual eigenfunctions that have $y$ volume integrals seem to be spreading out from the value of $\frac{1}{2}$. These behaviors are discussed numerically in the next section. 

The second plot shows the values of $I_l(m,n)$ for the .99 triangle, with the adjustment of the bounds of integration from $(0,1/2)$ to $(0,.495)$. Some of the structure of the plots is carried over from the $y$ volume integral case, but it is less coherent. Moreover, there seem to be subsequences whose bottom left side Neumann  data integrals are  approaching 1, which indicates that all of the Neumann data is congregating on one half of the bottom side. This suggests that even a lower bound for Neumann data on subsets of the boundary may not be possible, at least not for every sequence of eigenfunctions.  

We have verified that the two branches which are apparent in the $y$ volume integral plot are comprised of the same eigenfunctions whose bottom left Neumann integrals approach 1. Similar pictures for other triangles mentioned throughout this paper are in an appendix.

\section{Statistical Analysis of Eigenfunctions on Almost Isosceles Right  Triangles}

In this section, we introduce several new metrics for measuring how far a sequence of eigenfunctions is from having QE or QER type properties.  

\subsection{Introducing Running Averages}
Statements about quantum ergodicity allow for exceptional zero density subsequences. For the $y$ volume integral, we think of $1/2$ as signifying quantum ergodicity but, if the domain was truly ergodic, it is more accurate to state that every positive density subsequence needs to have a running average that converges to $1/2$. Or mathematically, for density 1 subsequence $u_{i_k}$:

\begin{equation}
    a_j = \frac{1}{j}\sum_{k=1}^j\int_T|h\partial_y u_{i_k}|^2 dV \rightarrow \frac{1}{2}.
\end{equation}

We can also say something similar about the proportion of Neumann data on a given side. Here, if the domain was indeed ergodic, for every subsequence of proportions, $I_l(m_j,n_j)$, with positive density we would have

\begin{equation}
   \frac{1}{j}\sum_{k=1}^j I_l(m_j,n_j) \rightarrow \frac{1}{2}.
\end{equation}
The same is of course true for any data defined similarly on subsets of the boundary. 
%

\subsection{Running Averages of Computed Runs}
We compute the running average of relevant metrics in an attempt to get a better understanding of the quantum ergodic properties of our triangle as we depart from the right isosceles case. The '$y$ volume integral' and 'Proportion Bottom Left' metrics are the ones used throughout this paper. A larger node count represents an increase in accuracy, but we found diminishing returns in increasing node counts in our numerics. As such, we considered 200 to be sufficient. Values were computed for the first 1250 eigenfunctions.

\begin{center}
 \begin{tabular}{||c c c c||} 
 \hline
 Triangle & Nodes & $y$ volume integral & Proportion of Bottom Left\\
 \hline\hline
 .99 & 425 & .4998 & .5083\\ 
 \hline
 .98 & 200 & .4996 & .5098\\ 
 \hline
 .97 & 200 & .4996 & .5103\\ 
 \hline
 .96 &200 & .4993 & .5097\\
 \hline
 .95 & 200 & .4992 & .5100\\ 
 \hline
\end{tabular}
\end{center}

The overall trend is consistent across metrics. The further we get from the right isosceles triangle, the farther  the metrics get from the values quantum ergodicity would predict. This is not enough evidence to suggest that these averages converge to a value other than what would be expected if the domain was ergodic, but it does heavily suggest that convergence is at least slower the farther away from isosceles the triangle is. 

\subsection{Percentage of Eigenfunctions Approach}

The issue with the methods previously described in this chapter is that they do not get to the heart of  what we want. Running averages can be influenced, especially at these frequencies, by density zero subsequences which are interesting but not definitive evidence that the domain itself is not ergodic. In service of trying to determine whether these experiments would cause us to expect a positive density subsequence that converges to an unexpected value, we instead shift our focus to percentages of eigenfunctions.

Statements about the density of sequences are extensions of the familiar discrete concept of percentages. They are statements about how common we would expect that particular subsequence to be. A density 1 subsequence, in the high-frequency limit, would be expected to appear for almost every value. As these are limits, there is substantial wiggle room. 

We can use this concept to develop metrics that could indicate whether positive density subsequences of the desired properties exist. Suppose we thought the running average of the $y$ volume integrals for the .99 triangle converged to a value less than .5. Then it would be sufficient to show for every finite $N$, some fixed $\epsilon > 0$, and some other fixed $\delta > 0$, that the percentage of the first $N$ eigenfunctions which have an $y$ volume integral less than $.5 - \epsilon$ is larger than $\delta$ for every $N$. If this condition was met, than the subsequence of all eigenfunctions whose $y$ volume integral is less than $.5 - \epsilon$ would have a density greater than $\delta$. This would show that the domain itself was not ergodic. 

Of course, there is nothing special about viewing the percentage of eigenfunctions below a certain threshold. Because we only need a subsequence of positive density, we can consider all eigenfunctions that have $y$ volume integrals sufficiently far away from .5. In the interest of having a metric that is equally valid regardless of the distribution of $y$ volume integral values, we consider the running percentage of eigenfunctions such that

\begin{equation}
    \Bigl|\int_T|h\partial_y u|^2 - .5\Bigr| > \epsilon
\end{equation}
for varying tolerances $\epsilon$. The values for a selection of runs are in the following table.

\begin{center}
 \begin{tabular}{||c| c c c||} 
 \hline
 Triangle & $\epsilon = .01$ & $\epsilon = .005$ & $\epsilon=.001$\\ [0.5ex] 
 \hline\hline
 .99 & 80.32 & 84.64 & 98.8\\ 
 \hline
 .98 & 82.9 & 93.0 & 99.3\\ 
 \hline
 .97 & 89.1 & 95.7 & 99.6\\ 
 \hline
 .96 & 91.0 & 95.6 & 99.04\\  [1ex] 
 \hline
 .95 & 92.7 & 96.6 & 99.52\\ 
 \hline
\end{tabular}
\end{center}

Perhaps more interesting than the exact numerical values are the trends. All of the graphs for all three thresholds for the .99, .98, .97, .96 and .95 triangles have the same fundamental shape: increasing with a vertical asymptote.

\begin{figure}
    \centering
    \includegraphics[scale=.5]{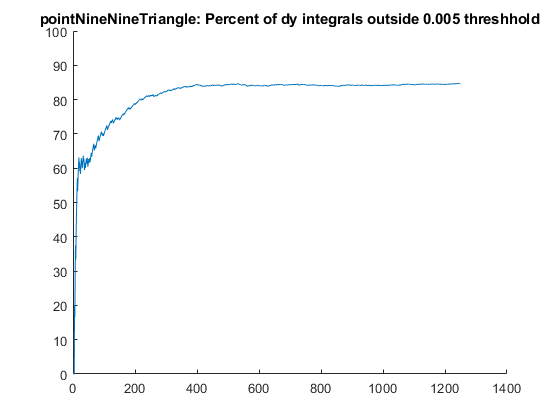}
    \caption{Running Percentage Graph -  Shows monotonic and asymptotic behavior. Similar shape holds for every triangle and rage in the previous table.}
    \label{fig:my_label}
\end{figure}

\subsection{Establishing Triangles with Different Behavior}
An interesting test case is the 30-60-90 triangle. Despite having the spatial  equidistribution property from being a rational planar polygon, it is known to be integrable. This triangle has a lot of symmetries, reflecting it over the y-axis gives the equilateral triangle for example, which is what leads to its integrability. By looking at triangles that are close to the 30-60-90, we can see how sensitive these numerics are. 

We ran two runs with a bottom length of $.575$ and $.58$. The bottom length of the 30-60-90 is $\frac{1}{\sqrt{3}}\approx .5774$, so these other triangles are close the the 30-60-90 but do not enjoy the geometric symmetries that have such a profound effect on the eigenfunctions. They produced the following results:

\begin{center}
 \begin{tabular}{||c| c c c||} 
 \hline
 Triangle & $\epsilon = .01$ & $\epsilon = .005$ & $\epsilon=.001$\\ [0.5ex] 
 \hline\hline
 .58 & 42.6 & 65.0 & 96.8 \\
 \hline
 30-60-90 & 11.7 & 16.4 & 30.9\\ 
 \hline
 .575 & 41.8 & 61.1 & 97.2\\ 
 \hline
\end{tabular}
\end{center}

Not only are the percentages noticeably lower than the other triangles, the shape of the running percentage scatter plot indicates that these numbers are decreasing significantly as the number of eigenvalues increases. This is the type of behavior that would be expected for an ergodic domain, but we see behaviors more in line with the previously discussed runs for the two triangles that are close to the 30-60-90. This complicates our interpretation, as we have a non-ergodic triangle that is displaying behavior that would be expected of an ergodic domain. 

\begin{figure}[h]
\begin{subfigure}{0.5\textwidth}
\includegraphics[width=1\linewidth, height=5cm]{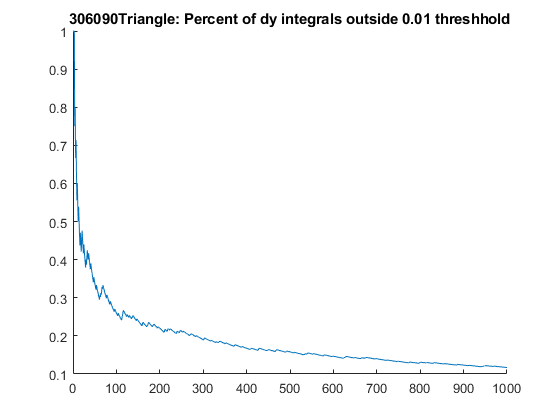} 
\label{fig:subim1}
\caption{30-60-90 - 450 Nodes - 1000 Eigenvalues}
\end{subfigure}
\begin{subfigure}{0.5\textwidth}
\includegraphics[width=1\linewidth, height=5cm]{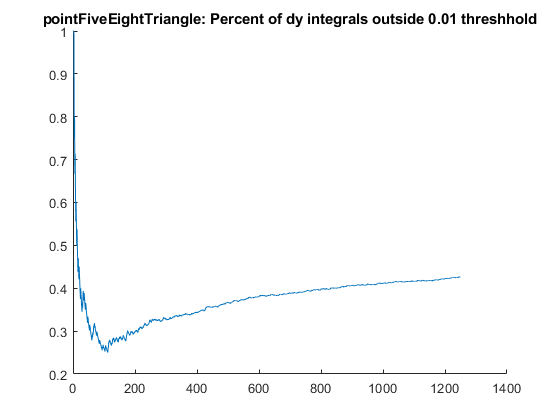}
\caption{.58 triangle - 200 Nodes - 1250 Eigenvalues}
\label{fig:subim2}
\end{subfigure}
\label{fig:image2}
\end{figure}


\section{Accuracy and Sanity Checks}

\newpage
\subsection{Mesh Convergence Test}
Confidence in our numerics increases if we can show convergence in accuracy metrics as  our mesh is refined. To test this, we chose two metrics: one for the eigenvalue and one for the eigenfunctions. 

The maximum eigenvalue difference is simply the largest difference between eigenvalues computed on the different meshes. The $L^2$ running average is the running average of the $L^2$ norm of the difference between the eigenfunctions computed on different meshes. To evaluate this difference, the higher accuracy function is interpolated on the coarser mesh. This adds another source of inaccuracy. 

We compared the 256 node calculations to the 128, 64, and 32 node calculations for the .99 triangle. The first 1000 eigenvalues and eigenfunctions were computed. The table below shows clear convergence on both metrics.

\begin{center}
 \begin{tabular}{|c| c| c|} 
 \hline
 Comparison & Max Eval Diff. & L2 Running Avg.\\
 \hline
 256 and 128 & 8.87 & .0091\\ 
 \hline
 256 and 64 & 122.23 & .0838\\ 
 \hline
 256 and 32 & 377.87 & .4762\\ 
 \hline
\end{tabular}
\end{center}

The 1000th Eigenvalue has a magnitude of around 30,000, so a maximum difference of 8.87 corresponds to about a .03\% difference. This shows we are not gaining a substantial amount of accuracy doubling the perimeter node count once we pass a certain threshold. This gives us confidence that our numerical experiment is well behaved. 

\subsection{Reported Errors}
FreeFEM itself can also report errors. It does this in 3 types, the relative error, absolute error, the backward error. All of these errors are generally increasing, so we will just report the error on the 1250th eigenvalue. 

\begin{center}
 \begin{tabular}{|c| c|} 
 \hline
 Error Type & Value\\
 \hline
 Absolute Error & 8.77e-8\\ 
 \hline
 Relative Error & 3.04e-15\\ 
 \hline
 Backwards Error & 7.292e-12\\ 
 \hline
\end{tabular}
\end{center}

\appendix
\section{Volume and Boundary Data for Near Isosceles Triangles}

\begin{figure}[h]
\begin{subfigure}{0.5\textwidth}
\includegraphics[width=1\linewidth, height=5cm]{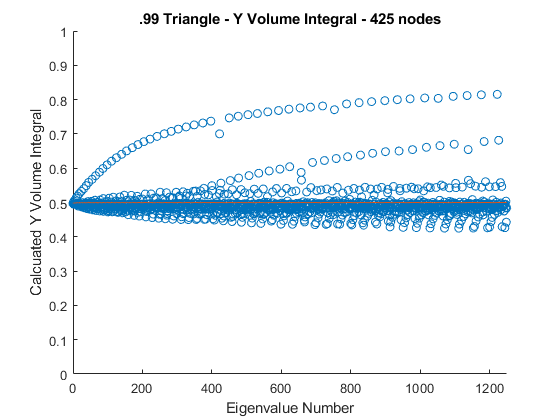} 
\label{fig:subim1}
\end{subfigure}
\begin{subfigure}{0.5\textwidth}
\includegraphics[width=1\linewidth, height=5cm]{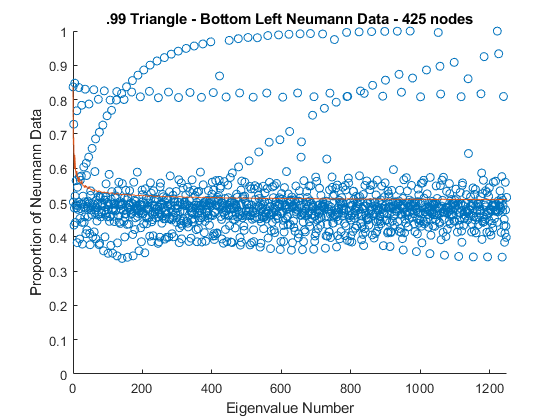}
\label{fig:subim2}
\end{subfigure}
\caption{.99 Triangle - 1250 Eigenvalues - Y Volume Integral and Bottom Left Neumann plots}
\label{fig:image2}
\end{figure}
\begin{figure}[h]
\begin{subfigure}{0.5\textwidth}
\includegraphics[width=1\linewidth, height=5cm]{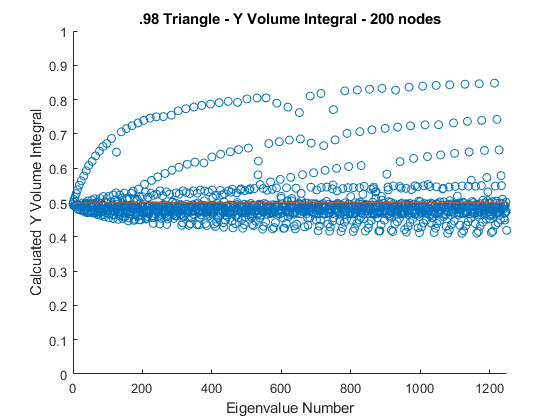} 
\label{fig:subim1}
\end{subfigure}
\begin{subfigure}{0.5\textwidth}
\includegraphics[width=1\linewidth, height=5cm]{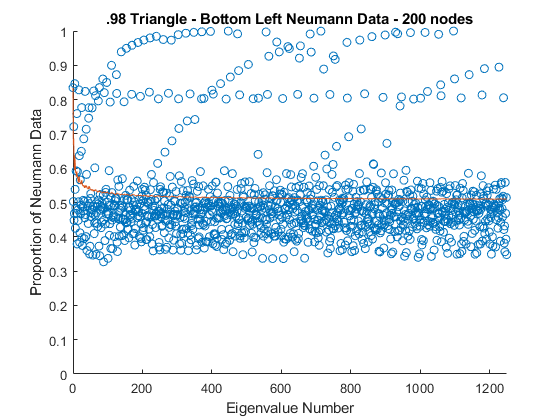}
\label{fig:subim2}
\end{subfigure}
\caption{.98 Triangle - 1250 Eigenvalues - Y Volume Integral and Bottom Left Neumann plots}
\label{fig:image2}
\end{figure}

\begin{figure}[h]
\begin{subfigure}{0.5\textwidth}
\includegraphics[width=1\linewidth, height=5cm]{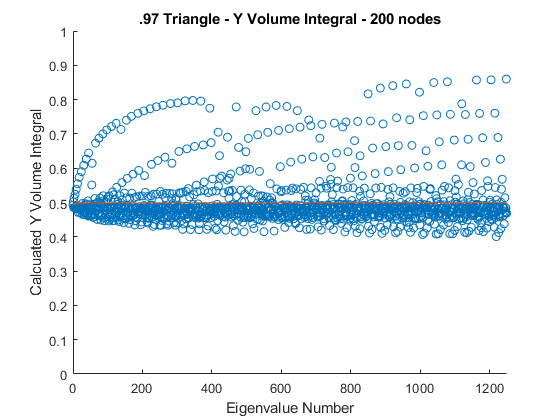} 
\label{fig:subim1}
\end{subfigure}
\begin{subfigure}{0.5\textwidth}
\includegraphics[width=1\linewidth, height=5cm]{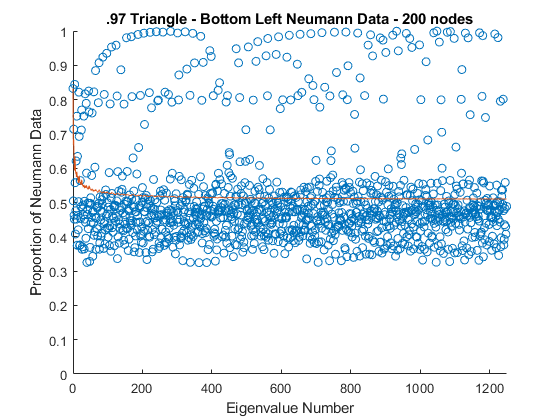}
\label{fig:subim2}
\end{subfigure}
\caption{.97 Triangle - 1250 Eigenvalues - Y Volume Integral and Bottom Left Neumann plots}
\label{fig:image2}
\end{figure}

\begin{figure}[h]
\begin{subfigure}{0.5\textwidth}
\includegraphics[width=1\linewidth, height=5cm]{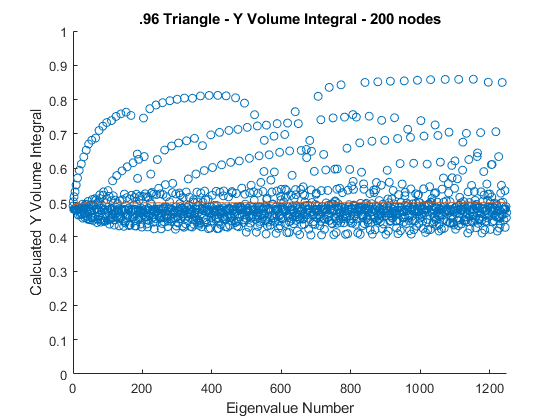} 
\label{fig:subim1}
\end{subfigure}
\begin{subfigure}{0.5\textwidth}
\includegraphics[width=1\linewidth, height=5cm]{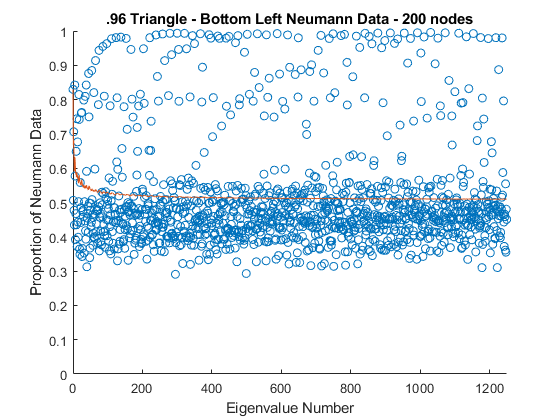}
\label{fig:subim2}
\end{subfigure}
\caption{.96 Triangle - 1250 Eigenvalues - Y Volume Integral and Bottom Left Neumann plots}
\label{fig:image2}
\end{figure}

\begin{figure}[h]
\begin{subfigure}{0.5\textwidth}
\includegraphics[width=1\linewidth, height=5cm]{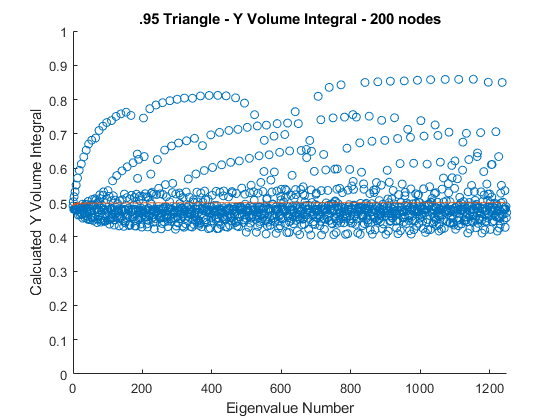} 
\label{fig:subim1}
\end{subfigure}
\begin{subfigure}{0.5\textwidth}
\includegraphics[width=1\linewidth, height=5cm]{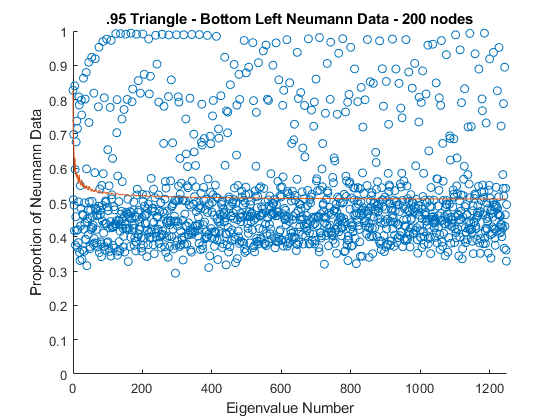}
\label{fig:subim2}
\end{subfigure}
\caption{.95 Triangle - 1250 Eigenvalues - Y Volume Integral and Bottom Left Neumann plots}
\label{fig:image2}
\end{figure}

 \clearpage
 
\newpage
\bibliographystyle{alpha}
\bibliography{main}
\newpage

\end{document}